\documentclass[a4paper, 10pt]{article}

\usepackage[latin1]{inputenc}
\usepackage{subfigure}
\usepackage{amsmath}
\usepackage{amssymb}
\usepackage{amsthm}
\usepackage{amsfonts}
\usepackage[cyr]{aeguill}
\usepackage{xspace}
\usepackage[english]{babel}
\usepackage{graphicx}
\usepackage{float}
\usepackage{mathtools}
\usepackage{psfrag,floatflt,hyperref}
\usepackage{verbatim}
\usepackage{geometry}
\usepackage{multirow}
\usepackage[retainorgcmds]{IEEEtrantools}
\usepackage{color}
\newtheorem{theorem}{Theorem}[section]

\newtheorem{proposition}[theorem]{Proposition}
\newtheorem{corollary}[theorem]{Corollary}

\title{Correlation of local densities of states on mesoscopic energy scales in random band matrices}
\author{Justine Louis\footnote{The author acknowledges support of the Swiss NSF through the SwissMAP grant and of the ERC through the RandMat grant.}}

\newcommand{\rd}{{\rm d}}
\begin{document}
\maketitle
\begin{abstract}
We are interested in the phase transition of the correlation function of local densities of states at mesoscopic scales of random band matrices of width $W$ in dimension $2$. As a result, we show that the local densities of states are alternately positively and negatively correlated in the diffusive regime $O(\log(L/W))$ times, $L$ being the size of the system.
\end{abstract}

\section{Introduction}
Impurities in a disordered quantum system can be modelised by a random band Hermitian matrix $H$ corresponding to the Hamiltonian of the system. Pure materials have an underlying lattice structure which can be represented by a finite lattice in $\mathbb{Z}^d$, say the torus $\mathbb{T}=\left[-\frac{L}{2}\right.,\left.\frac{L}{2}\right)^d\cap\mathbb{Z}^d$. Given $x,y\in\mathbb{T}$, corresponding to the lattice sites, the matrix $H_{xy}$ represents the quantum system in a $d$-dimensional discrete box of length $L$. Following the model introduced by Erd{\H{o}}s and Knowles in \cite{erdHos2015altshulerI,erdHos2015altshulerII}, we assume that the matrix $H$ is an Hermitian matrix whose upper triangular entries are independent random variables with zero mean.
Let $S_{xy}:=\mathbb{E}\lvert H_{xy}\rvert^2$ be a deterministic matrix given by an arbitrary profile function $f$ on the scale $W$, that is,
\begin{equation*}
S_{xy}=\frac{1}{M-1}f\left(\frac{\left[x-y\right]_L}{W}\right),\quad M:=\sum_{x\in\mathbb{T}}f\left(\frac{x}{W}\right)
\end{equation*}
where $\left[x\right]_L$ denotes the canonical representative of $x\in\mathbb{Z}^d$ in the torus $\mathbb{T}$ and $f:\mathbb{R}^d\rightarrow\mathbb{R}$ is an even, bounded, non-negative, piecewise $C^1$ function such that $f$ and $\lvert\nabla f\rvert$ are integrable and $\int_{\mathbb{R}^d}{\rm d}xf(x)\lvert x\rvert^{4+c_1}<\infty$ for some constant $c_1>0$. Thus, for all $x\in\mathbb{T}$,
\begin{equation*}
\sum_{y\in\mathbb{T}}S_{xy}=\frac{M}{M-1}=:\mathcal{I}.
\end{equation*}
We assume that the law of $H_{xy}$ is symmetric, i.e. $H_{xy}$ and $-H_{xy}$ have the same law and that $A_{xy}:=(S_{xy})^{-1/2}H_{xy}$ have uniform subexponential decay, that is, $$\mathbb{P}(\lvert A_{xy}\rvert>\xi)\leqslant c_2e^{-\xi^{c_3}}$$ for some constants $c_2,c_3>0$ and for any $\xi>0$.
Let $\phi:\mathbb{R}\rightarrow\mathbb{R}$ be a test function defined as a smooth function verifying the following conditions \textbf{(C)}: $\int_{\mathbb{R}}\phi(E)\rd E=2\pi$ and such that for every $q>0$, there exists a constant $C_q$ satisfying
\begin{equation*}
\lvert\phi(E)\rvert\leqslant\frac{C_q}{1+\lvert E\rvert^q}.
\end{equation*}
We are interested in the correlation of the number of eigenvalues around two energies $E_1<E_2$ such that $\omega=E_2-E_1$ is much larger than the energy window $\eta$. More precisely, we are concerned with the following correlation
\begin{equation*}
\frac{\langle Y_{\phi_1}^\eta(E_1)\ \!;Y_{\phi_2}^\eta(E_2)\rangle}{\langle Y_{\phi_1}^\eta(E_1)\rangle\langle Y_{\phi_2}^\eta(E_2)\rangle}\quad\textrm{where }\langle X\rangle:=\mathbb{E}X,\quad\langle X\ \!;Y\rangle:=\mathbb{E}(XY)-\mathbb{E}X\mathbb{E}Y
\end{equation*}
and $Y_{\phi_i}^\eta(E)$ is the smoothed local density of states around energy $E$ on the scale $\eta$, defined by
\begin{equation*}
Y_{\phi_i}^\eta(E):=\frac{1}{L^d}\ \!\textrm{Tr}\ \!\phi_i^\eta(H/2-E),\ i=1,2,
\end{equation*}
$\phi_i^\eta$ being the rescaled test function $\phi_i^\eta(E):=\eta^{-1}\phi_i(\eta^{-1}E)$ and $\phi_i$ a test function satisfying conditions \textbf{(C)}. On the mesoscopic energy scale which corresponds to energy scales much larger than the eigenvalue spacing and much smaller than the total macroscopic energy scale of the system, we observe a phase transition at the critical energy, the Thouless energy, given by $\eta_c=W^2/L^2$.

Throughout the paper we make the following assumptions 
\begin{equation}
\label{hyp}
\omega\gg\eta,\quad\! W\ll L,\quad\! L\leqslant W^C,\quad\! \eta\ll1,\quad\! \eta\gg M^{-1/3},\quad\! E_1,E_2\in\left[-1+\kappa,1-\kappa\right],\quad\!\omega\leqslant c_\ast
\end{equation}
for some constant $C$, $\kappa$ a fixed positive constant, $c_\ast$ a small enough positive constant depending on $\kappa$. We introduce the covariance matrix of $S_{x0}$, its fourth moment and the covariance matrix of $f$
\begin{equation*}
D_{ij}:=\frac{1}{2}\sum_{x\in\mathbb{T}}\frac{x_ix_j}{W^2}S_{x0},\quad Q:=\frac{1}{32}\sum_{x\in\mathbb{T}}S_{x0}\bigg|D^{-1/2}\frac{x}{W}\bigg|^4,\quad (D_0)_{ij}:=\frac{1}{2}\int_{\mathbb{R}^d}\rd x\,x_ix_jf(x)
\end{equation*}
and assume that $c_4<D_0<c_5$ in the sense of quadratic forms for some positive constants $c_4,c_5$. Here we choose the matrix $D$ as a positive multiple of the $d\times d$ identity matrix $I_d$, i.e. $D=\mathcal{D}I_d$, $\mathcal{D}>0$. We also introduce the parameters
\begin{align*}
&\alpha:=e^{i(\arcsin(E_1+i\eta)-\arcsin(E_2-i\eta))},\quad u:=\lvert1-\alpha\rvert,\quad\zeta\in\mathbb{S}^1\textrm{ such that }1-\alpha=:u\zeta\\
&b:=\frac{1}{\mathcal{D}}\left(\frac{\sqrt{u}L}{2\pi W}\right)^2,\quad R:=\frac{\epsilon L}{2\pi W},\textrm{ where }\epsilon>0\textrm{ will be defined later}.
\end{align*}
Let $\nu\equiv\nu(E)=2\sqrt{1-E^2}/\pi$ and $E:=(E_1+E_2)/2$. Then $u$ and $\zeta$ expand as
\begin{equation}
\label{zeta}
u=\frac{2\omega}{\pi\nu}(1+O(\eta^2/\omega^2+\omega)),\quad\zeta=i+\frac{\omega}{\pi\nu}+2\frac{\eta}{\omega}+O\left(\omega^2+\eta+\frac{\eta^2}{\omega^2}\right).
\end{equation}
The diffusive regime is defined by $\eta\gg\eta_c$ which corresponds to large samples, i.e. $L$ is large with respect to the diffusion length $W/\sqrt{\eta}$. In terms of our parameter $b$ behaving as $\omega L^2/W^2$ using (\ref{zeta}), it means that $b\gg1$ from assumption~(\ref{hyp}) $\omega\gg\eta$. On the contrary, the mean-field regime is defined by $\eta\ll\eta_c$ corresponding to small samples.

The main contribution of the present paper is an exact asymptotic expression for the local density-density correlation function in dimension $1$ valid in both regimes while in dimension $2$ we provide a more precise expression of it with respect to the asymptotics derived in \cite{erdHos2015altshulerII} in the diffusive regime improving \cite[Proposition 3.5]{erdHos2015altshulerII}. In their paper the authors estimate a series appearing in the main term using a Riemann sum estimate while the key point here consists in no longer making this approximation. A curious fact is that we observe that for large $b$ but less than $(\log(L/W))^2$ the correlation function has an interesting oscillatory behaviour. The following theorem states our results.
\begin{theorem}
\label{Th}
There exists a constant $c_0>0$ such that the local density-density correlation satisfies:\\
(i) In dimension $1$ for all $b>0$,
\begin{align*}
&\frac{\langle Y_{\phi_1}^\eta(E_1)\ \!;Y_{\phi_2}^\eta(E_2)\rangle}{\langle Y_{\phi_1}^\eta(E_1)\rangle\langle Y_{\phi_2}^\eta(E_2)\rangle}=\\
&-\frac{1}{16(\pi\nu)^{5/2}\sqrt{\mathcal{D}}\omega^{3/2}LW}\left(\frac{\sinh(\pi\sqrt{2b})+\sin(\pi\sqrt{2b})}{\sinh^2(\pi\sqrt{b/2})+\sin^2(\pi\sqrt{b/2})}+\pi\sqrt{2b}\frac{\sinh(\pi\sqrt{2b})\sin(\pi\sqrt{2b})}{(\sinh^2(\pi\sqrt{b/2})+\sin^2(\pi\sqrt{b/2}))^2}\right)\\
&+O\left(\frac{\omega^{-1/2}}{LW}\left(1+\frac{\eta}{\omega^2}\right)+\frac{1}{\omega L^2}\left(1+\frac{\eta}{\omega^2}\right)+\frac{1}{\omega W^3}e^{-\pi\sqrt{2\omega}L/W}+\frac{W^{-c_0-1}}{L(\omega+\eta)^{1/2}}+\frac{W^{-c_0}}{L^2(\omega+\eta)}\right).
\end{align*}
(ii) In dimension $2$ for $b\gg1$,
\begin{align*}
&\frac{\langle Y_{\phi_1}^\eta(E_1)\ \!;Y_{\phi_2}^\eta(E_2)\rangle}{\langle Y_{\phi_1}^\eta(E_1)\rangle\langle Y_{\phi_2}^\eta(E_2)\rangle}=-\frac{1}{2\pi^5\mathcal{D}\nu^4L^2W^2}\left[\frac{L^2}{\pi \mathcal{D}W^2}e^{-\pi\sqrt{2b}}b^{-3/4}\sin\left(\pi\sqrt{2b}-\frac{\pi}{8}\right)-(Q-1)\lvert\log\omega\rvert\right.\\
&\left.+O\left(1+\frac{\eta^2}{\omega^3}+\omega\lvert\log\omega\rvert+e^{-\pi\sqrt{2b}}b^{1/4}\left(1+\frac{1}{\omega W^2}+\frac{L^2}{W^2}b^{-3}\right)+\frac{1}{W^{c_0}}\lvert\log(\omega+\eta)\rvert+\frac{W^{2-c_0}}{L^2(\omega+\eta)}\right)\right].
\end{align*}
\end{theorem}
We emphasise that for the two-dimensional case, the second term has been calculated in \cite{erdHos2015altshulerII} while the first term is new, and is dominant for $b\ll(\log(L/W))^2$, showing that the correlation function oscillates around zero $O(\log(L/W))$ times. Then for $b\gtrsim(\log(L/W))^2$ it is dominated by the logarithmic term. This will be shown in Proposition~\ref{sameorder} and Corollary~\ref{alt} below.

An interesting question would be to see whether this oscillatory behaviour is as well observed in other systems and its physical interpretation.
\section{Calculation of the correlation function}
In \cite[Theorem $6.1$]{erdHos2015altshulerII}, it has been shown that, under the assumptions~(\ref{hyp}), there exists a constant $c_0>0$ such that for any $E_1,E_2\in\left[-1+\kappa,1-\kappa\right]$ for small enough $c_\ast>0$, the local density-density correlation satisfies
\begin{equation}
\label{corr}
\frac{\langle Y_{\phi_1}^\eta(E_1)\ \!;Y_{\phi_2}^\eta(E_2)\rangle}{\langle Y_{\phi_1}^\eta(E_1)\rangle\langle Y_{\phi_2}^\eta(E_2)\rangle}=\frac{1}{(LW)^d}\left(\Theta_{\phi_1,\phi_2}^\eta(E_1,E_2)+M^{-c_0}O\left(R(\omega+\eta)+\frac{M}{N(\omega+\eta)}\right)\right)
\end{equation}
where $R(s):=1+\mathbf{1}(d=1)s^{-1/2}+\mathbf{1}(d=2)\lvert\log s\rvert$ and the leading term $\Theta_{\phi_1,\phi_2}^\eta$ is given by 
\begin{equation}
\label{theta}
\Theta^\eta_{\phi_1,\phi_2}(E_1,E_2)=\frac{2W^d}{\pi^4\nu^4L^d}\mathrm{Re}\,\mathrm{Tr}\frac{S}{(1-\alpha S)^2}\left(1+O(\omega)\right)+O(1).
\end{equation}
The proof of the above relation is given in the appendix. For $q\in\left[-\pi W,\pi W\right)^d$ let $\widehat{S}_W(q)$ denote $\widehat{S}(q/W)$ where $\widehat{S}(p):=\sum_{x\in\mathbb{T}}e^{-ip\cdot x}S_{x0}$ is defined for all $p\in\left[-\pi,\pi\right)^d$. Also define the following quantity
\begin{equation*}
\mathcal{Q}(q):=\frac{1}{4!}\sum_{x\in\mathbb{T}}\left(x\cdot q/W\right)^4S_{x0}.
\end{equation*}
Recall from \cite[Lemma B.$1$]{erdHos2015altshulerII} that for any $\epsilon>0$ there is a $\delta_\epsilon>0$ such that $\widehat{S}_W(q)$ is bounded, $|\widehat{S}_W(q)|\leqslant1-\delta_\epsilon$ if $|q|\geqslant\epsilon$ for large enough $W$, and has the following expansion
\begin{equation*}
\widehat{S}_W(q)=\mathcal{I}-\mathcal{D}\lvert q\rvert^2+\mathcal{Q}(q)+O(|q|^{4+c_1})
\end{equation*}
which comes from a fourth order Taylor's expansion of
\begin{equation*}
\mathcal{I}-\widehat{S}_W(q)=\frac{1}{M-1}\sum_{v\in W^{-1}\mathbb{T}}(1-\cos(q\cdot v))f(v).
\end{equation*}
Let $\epsilon>0$ be such that $\widehat{S}_W(q)=\mathcal{I}-a(q)$ where $a$ is a function satisfying $c_4\lvert q\rvert^2\leqslant a(q)\leqslant c_5\lvert q\rvert^2\leqslant1$ for $\lvert q\rvert\leqslant\epsilon$.
Equation \cite[B.15]{erdHos2015altshulerII} states that
\begin{align}
\label{tr}
\mathrm{Tr}\frac{S}{(1-\alpha S)^2}&=\sum_{q\in W\mathbb{T}^*}\frac{\widehat{S}_W(q)}{(1-\alpha\widehat{S}_W(q))^2}\mathbf{1}(|q|\leqslant\epsilon)+O\left(\frac{L^d}{\delta_\epsilon W^d}\right)
\end{align}
where $\mathbb{T}^*=\frac{2\pi}{L}\mathbb{T}$.

In Subsection~\ref{dim1} we study the above trace in the one-dimensional case while the two-dimensional case is treated in Subsection~\ref{dim2}.
\subsection{Dimension 1}
\label{dim1}
\begin{proposition}
\label{propTr1}
In dimension $1$, the following relation holds
\begin{align*}
\mathrm{Tr}\frac{S}{(1-\alpha S)^2}=&\frac{1}{32\pi^3\mathcal{D}^2}\frac{L^4}{W^4}\frac{1}{(\zeta b)^{3/2}}\left(\coth\left(\pi\sqrt{\zeta b}\right)+\pi\sqrt{\zeta b}\sinh^{-2}\left(\pi\sqrt{\zeta b}\right)\right)\\
&+O\left(\frac{L}{W}\omega^{-1/2}+\frac{L}{W^2}\omega^{-5/2}+\frac{W^2}{\omega^5L^3}+\frac{L^2}{W^3\omega^2}e^{-\pi\sqrt{2\omega}L/W}\right).
\end{align*}
\end{proposition}
\begin{proof}
With the notations introduced above, equation \cite[B.16]{erdHos2015altshulerII} reads
\begin{align*}
\sum_{q\in W\mathbb{T}^*}\frac{\widehat{S}_W(q)}{(1-\alpha\widehat{S}_W(q))^2}\mathbf{1}(|q|\leqslant\epsilon)&=\mathcal{I}\sum_{q\in W\mathbb{T}^*}\frac{\mathbf{1}(|q|\leqslant\epsilon)}{(1-\alpha\widehat{S}_W(q))^2}+O\left(\frac{L}{W}R(u)\right)
\end{align*}
and equation \cite[B.17]{erdHos2015altshulerII}
\begin{align*}
\sum_{q\in W\mathbb{T}^*}\frac{\mathbf{1}(|q|\leqslant\epsilon)}{(1-\alpha\widehat{S}_W(q))^2}&=\sum_{q\in W\mathbb{T}^*}\frac{\mathbf{1}(|q|\leqslant\epsilon)}{(1-\alpha+\mathcal{D}\lvert q\rvert^2)^2}+\frac{L}{W}O\left(R(u)+\frac{u^{-5/2}}{W}\right)
\end{align*}
which together with equation~(\ref{tr}) lead to the following expression for the trace appearing in the $\Theta_{\phi_1,\phi_2}^\eta(E_1,E_2)$ term in (\ref{theta})
\begin{equation}
\label{trace1}
\mathrm{Tr}\frac{S}{(1-\alpha S)^2}=\mathcal{I}\sum_{q\in W\mathbb{T}^*}\frac{\mathbf{1}(|q|\leqslant\epsilon)}{(1-\alpha+\mathcal{D}\lvert q\rvert^2)^2}+O\left(\frac{L}{W}u^{-1/2}+\frac{L}{W^2}u^{-5/2}\right).
\end{equation}
Let
\begin{equation*}
\mathcal{S}_1(b,R):=\sum_{n\in\mathbb{Z}}\frac{\mathbf{1}(\lvert n\rvert\leqslant R)}{(n^2+\zeta b)^2},
\end{equation*}
then
\begin{equation}
\label{s1}
\sum_{q\in W\mathbb{T}^*}\frac{\mathbf{1}(|q|\leqslant\epsilon)}{(1-\alpha+\mathcal{D}\lvert q\rvert^2)^2}=\frac{1}{\mathcal{D}^2}\left(\frac{L}{2\pi W}\right)^4\mathcal{S}_1(b,R).
\end{equation}
In dimension $1$, 
\begin{equation*}
\mathcal{S}_1(b,R)=\mathcal{S}_1(b)+O\left(R^{-3}\right),\textrm{ with }\mathcal{S}_1(b):=\sum_{n\in\mathbb{Z}}\frac{1}{(n^2+\zeta b)^2}.
\end{equation*}
Using Poisson summation formula,
\begin{equation}
\label{sum1}
\sum_{q\in\mathbb{Z}}\frac{1}{z^2+q^2}=\sum_{n\in\mathbb{Z}}\int_\mathbb{R}{\rm d}q\ \frac{e^{-2\pi inq}}{z^2+q^2}=\sum_{n\in\mathbb{Z}}\frac{\pi}{z}e^{-2\pi|n|z}=\frac{\pi}{z}\coth(\pi z).
\end{equation}
Thus
\begin{equation*}
\sum_{q\in\mathbb{Z}}\frac{1}{z+q^2}=\frac{\pi}{\sqrt{z}}\coth(\pi\sqrt{z}).
\end{equation*}
By differentiating the above relation with respect to $z$, we obtain
\begin{equation*}
\sum_{q\in\mathbb{Z}}\frac{1}{(z+q^2)^2}=\frac{\pi}{2z^{3/2}}\coth(\pi\sqrt{z})+\frac{\pi^2}{2z}\frac{1}{\sinh^2(\pi\sqrt{z})}.
\end{equation*}
Hence
\begin{equation*}
\mathcal{S}_1(b)=\frac{\pi}{2(\zeta b)^{3/2}}\coth(\pi\sqrt{\zeta b})+\frac{\pi^2}{2\zeta b}\frac{1}{\sinh^2(\pi\sqrt{\zeta b})}.
\end{equation*}
The result then follows from (\ref{s1}) and (\ref{trace1}) and the fact that $\mathcal{I}=1+O(W^{-1})$.
\end{proof}
\begin{proof}[Proof of Theorem~\ref{Th}(i)]
We have
\begin{align*}
&\coth(\pi\sqrt{ib})=\frac{1}{2}\frac{\sinh(\pi\sqrt{2b})-i\sin(\pi\sqrt{2b})}{\sinh^2(\pi\sqrt{b/2})+\sin^2(\pi\sqrt{b/2})}\\
&\frac{1}{\sinh^2(\pi\sqrt{ib})}=\frac{\sinh^2(\pi\sqrt{b/2})\cos(\pi\sqrt{2b})-\sin^2(\pi\sqrt{b/2})-i\sin(\pi\sqrt{2b})\sinh(\pi\sqrt{2b})/2}{(\sinh^2(\pi\sqrt{b/2})+\sin^2(\pi\sqrt{b/2}))^2}.
\end{align*}
Hence
\begin{equation*}
\mathrm{Re}\,\mathcal{S}_1(b)=-\frac{\pi}{4\sqrt{2}}\frac{1}{b^{3/2}}\frac{\sinh(\pi\sqrt{2b})+\sin(\pi\sqrt{2b})}{\sinh^2(\pi\sqrt{b/2})+\sin^2(\pi\sqrt{b/2})}-\frac{\pi^2}{4b}\frac{\sinh(\pi\sqrt{2b})\sin(\pi\sqrt{2b})}{(\sinh^2(\pi\sqrt{b/2})+\sin^2(\pi\sqrt{b/2}))^2}.
\end{equation*}
The result follows using Proposition~\ref{propTr1}.
\end{proof}
Below is a plot of the following function at the transition $b\simeq1$
\begin{equation*}
f(b)=-\frac{\sinh(\pi\sqrt{2b})+\sin(\pi\sqrt{2b})}{\sinh^2(\pi\sqrt{b/2})+\sin^2(\pi\sqrt{b/2})}-\pi\sqrt{2b}\frac{\sinh(\pi\sqrt{2b})\sin(\pi\sqrt{2b})}{(\sinh^2(\pi\sqrt{b/2})+\sin^2(\pi\sqrt{b/2}))^2}.
\end{equation*}
\begin{figure}[H]
\centering
\includegraphics[width=8cm]{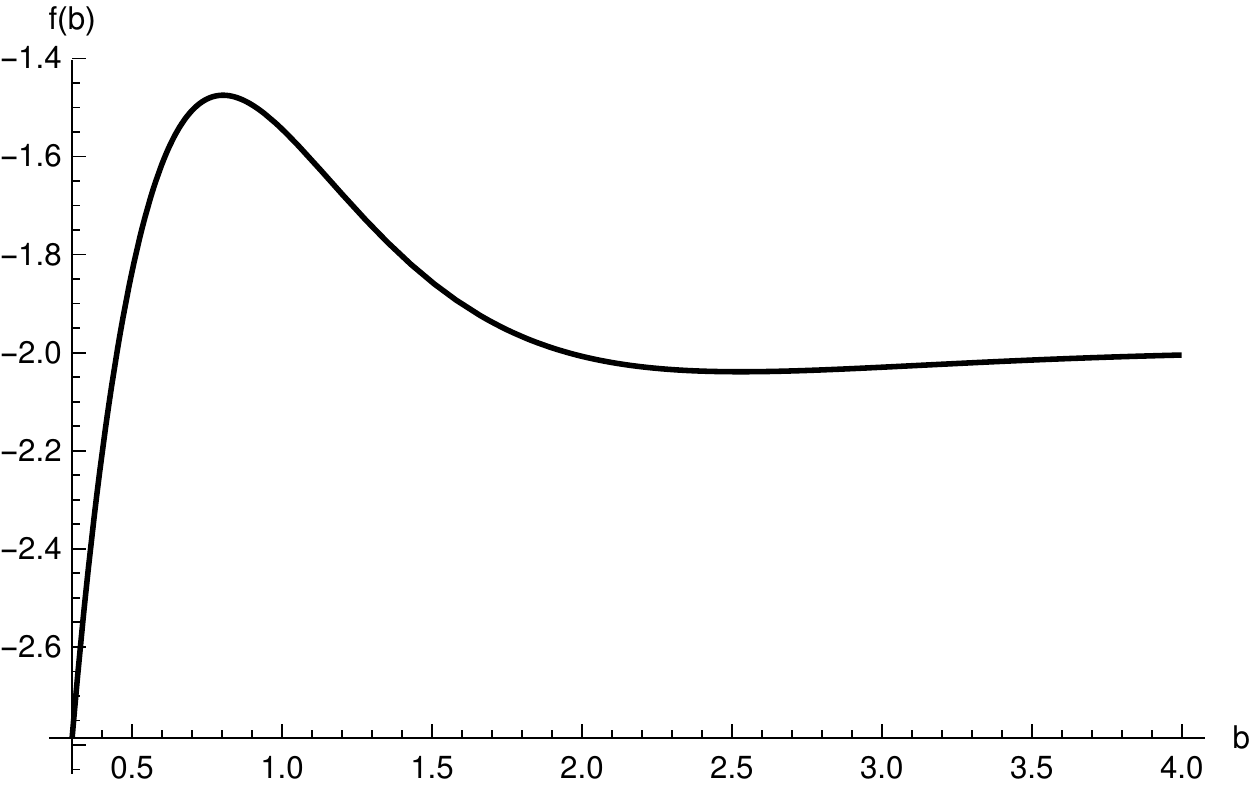}
\caption{The correlation function multiplied by $\omega^{3/2}$ at the phase transition as a function of $\omega(L/W)^2$.}
\end{figure}
\begin{corollary}
In the mean-field regime $b\ll1$,
\begin{align*}
\frac{\langle Y_{\phi_1}^\eta(E_1)\ \!;Y_{\phi_2}^\eta(E_2)\rangle}{\langle Y_{\phi_1}^\eta(E_1)\rangle\langle Y_{\phi_2}^\eta(E_2)\rangle}=&-\frac{1}{LW}\left[\frac{1}{2\pi^2\nu^2\omega^2}\frac{W}{L}-\frac{1}{360\mathcal{D}^2\pi^4\nu^4}\left(\frac{L}{W}\right)^3\right.\\
&\left.+O\left(\frac{1}{\sqrt{\omega}}+\frac{1}{\omega^{5/2}W}+\frac{1}{\omega^2L}+\frac{\eta^2}{\omega^4}\frac{W}{L}+\frac{W^{-c_0}}{(\omega+\eta)^{1/2}}+\frac{W}{L(\omega+\eta)}\right)\right]
\end{align*}
which recovers \cite[Theorem $2.9$ $(ii)$]{erdHos2015altshulerI} up to a multiplicative constant in the leading term and gives the next term in the asymptotic expansion.
In the diffusive regime $1\ll b\ll\min\left((\log(L/W))^2,(\eta/\omega+\omega)^{-2}\right)$,
\begin{align*}
\frac{\langle Y_{\phi_1}^\eta(E_1)\ \!;Y_{\phi_2}^\eta(E_2)\rangle}{\langle Y_{\phi_1}^\eta(E_1)\rangle\langle Y_{\phi_2}^\eta(E_2)\rangle}=&-\frac{1}{16\sqrt{2}\pi^7\nu^4\mathcal{D}^2}\frac{L^2}{W^4} b^{-3/2}\Bigg[1+4\sqrt{2}e^{-\pi\sqrt{2b}}\left[\sqrt{b}\sin(\pi\sqrt{2b})\right.\\
&\left.+O\left(1+\sqrt{\omega}\frac{L}{W^2}+b\left(\frac{\eta}{\omega}+\omega\right)+b^{3/2}e^{\pi\sqrt{2b}}\frac{W^2}{L^4}\left(\frac{W^{-c_0}}{(\omega+\eta)^{1/2}}+\frac{W}{L(\omega+\eta)}\right)\right)\right]\Bigg]
\end{align*}
which recovers \cite[Theorem $2.12$ $(i)$]{erdHos2015altshulerII} and gives the next term in the asymptotic expansion.
\end{corollary}
\begin{proof}
In the mean-field regime $b\ll1$, $f(b)=-\frac{4\sqrt{2}}{\pi\sqrt{b}}+\frac{4\sqrt{2}}{45}\pi^3b\sqrt{b}+O(b^{7/2})$ so that the following asymptotic formula holds
\begin{equation*}
\mathrm{Tr}\frac{S}{(1-\alpha S)^2}=\frac{1}{\mathcal{D}^2}\left(\frac{L}{2\pi W}\right)^4\left(\frac{1}{(\zeta b)^2}+\frac{\pi^4}{45}+O\left(b+\frac{1}{b^2W}+\frac{W^3}{L^3}u^{-1/2}+\frac{W^2}{L^3}u^{-5/2}\right)\right).
\end{equation*}
In the diffusive regime $1\ll b\ll\min\left((\log(L/W))^2,(\eta/\omega+\omega)^{-2}\right)$,\\
$f(b)=-2-8\sqrt{2}\pi e^{-\pi\sqrt{2b}}\sqrt{b}\sin(\pi\sqrt{2b})+O(e^{-\pi\sqrt{2b}})$, so that the following asymptotic formula holds
\begin{equation*}
\mathrm{Tr}\frac{S}{(1-\alpha S)^2}=\frac{1}{32\pi^3\mathcal{D}^2}\frac{L^4}{W^4}\frac{1}{(\zeta b)^{3/2}}\left(1+4\pi e^{-2\pi\sqrt{\bar{\zeta}b}}\left(\sqrt{\zeta b}+O\left(1+\sqrt{\omega}\frac{L}{W^2}+e^{2\pi\sqrt{\bar{\zeta}b}}\left(u+(uW)^{-1}\right)\right)\right)\right).
\end{equation*}
The results follow from (\ref{theta}) and (\ref{corr}).
\end{proof}

\subsection{Dimension 2}
\label{dim2}
From equation~(\ref{tr}), it follows
\begin{align*}
\mathrm{Tr}\frac{S}{(1-\alpha S)^2}&=\sum_{q\in W\mathbb{T}^*}\frac{\widehat{S}_W(q)}{(1-\alpha\widehat{S}_W(q))^2}\mathbf{1}(|q|\leqslant\epsilon)+O\left(\frac{L^2}{\delta_\epsilon W^2}\right)\\
&=\sum_{q\in\frac{2\pi W}{L}\mathbb{Z}^2}\frac{\mathcal{I}-\mathcal{D}\lvert q\rvert^2}{(1-\alpha\widehat{S}_W(q))^2}\mathbf{1}(|q|\leqslant\epsilon)+O\left(\frac{L^2}{W^2}\right)\\
&=\sum_{q\in\frac{2\pi W}{L}\mathbb{Z}^2}\frac{\mathcal{I}-\mathcal{D}\lvert q\rvert^2}{(1-\alpha+\mathcal{D}\lvert q\rvert^2-\mathcal{Q}(q))^2}\mathbf{1}(|q|\leqslant\epsilon)+O\left(\frac{L^2}{W^2}\right)
\end{align*}
where in the second equality we used the fact that
\begin{equation*}
\sum_{q\in W\mathbb{T}^\ast}\frac{\mathcal{Q}(q)}{\left(1-\alpha\widehat{S}_W(q)\right)^2}\mathbf{1}(\lvert q\rvert\leqslant\epsilon)=O\Bigg(\sum_{q\in W\mathbb{T}^\ast}\frac{\lvert q\rvert^4}{\left(1-\alpha\widehat{S}_W(q)\right)^2}\mathbf{1}(\lvert q\rvert\leqslant\epsilon)\Bigg)=O\left(\frac{L^2}{W^2}\right).
\end{equation*}
Expanding the denominator yields to
\begin{equation*}
\sum_{q\in\frac{2\pi W}{L}\mathbb{Z}^2}\frac{\mathbf{1}(|q|\leqslant\epsilon)}{(1-\alpha+\mathcal{D}\lvert q\rvert^2-\mathcal{Q}(q))^2}=\sum_{q\in\frac{2\pi W}{L}\mathbb{Z}^2}\frac{\mathbf{1}(|q|\leqslant\epsilon)}{(\mathcal{D}\lvert q\rvert^2+u\zeta)^2}\left(1+2\frac{\mathcal{Q}(q)}{\mathcal{D}\lvert q\rvert^2+u\zeta}+O\left(\left(\frac{|q|^4}{|q|^2+u\zeta}\right)^2\right)\right).
\end{equation*}
In \cite{erdHos2015altshulerII}, the authors estimated this trace by using a Riemann sum approximation for each summation. The real part of the first term then vanishes while the next terms give a logarithmic contribution. Here we no longer make the Riemann sum approximation and give a precise estimation of this term. More precisely, rewriting the first summation over the torus $(2\pi W/L)\mathbb{Z}^2$ as a summation over $\mathbb{Z}^2$, it can be expressed in terms of an inhomogeneous Epstein zeta function
\begin{align*}
\mathcal{J}:=\sum_{q\in\frac{2\pi W}{L}\mathbb{Z}^2}\frac{\mathbf{1}(|q|\leqslant\epsilon)}{(\mathcal{D}\lvert q\rvert^2+u\zeta)^2}=\frac{1}{\mathcal{D}^2}\left(\frac{L}{2\pi W}\right)^4\mathcal{S}_2(b,R)
\end{align*}
where
\begin{equation*}
\mathcal{S}_2(b,R):=\sum_{n\in\mathbb{Z}^2}\frac{\mathbf{1}(|n|\leqslant R)}{(\lvert n\rvert^2+\zeta b)^2}.
\end{equation*}
To express the leading term in terms of $\mathcal{S}_2(b)$ defined by
\begin{equation*}
\mathcal{S}_2(b):=\sum_{n\in\mathbb{Z}^2}\frac{1}{(\lvert n\rvert^2+ib)^2}
\end{equation*}
we expand $\zeta$ using (\ref{zeta}) implying that
\begin{align*}
\mathcal{J}=\frac{1}{\mathcal{D}^2}\left(\frac{L}{2\pi W}\right)^4\left(\mathcal{S}_2(b)+O\left(\left(\sum_{n\in\mathbb{Z}^2}\frac{1}{(|n|^2+ib)^3}+R^{-4}\right)\left(b(\omega+\eta+\eta^2/\omega^2)\right)+\frac{u}{R^2}\right)\right).
\end{align*}
We have
\begin{equation*}
\sum_{n\in\mathbb{Z}^2}\frac{1}{(|n|^2+ib)^3}=\frac{1}{b^3}\sum_{m\in b^{-1/2}\mathbb{Z}^2}\frac{1}{(m^2+i)^3}=\frac{2\pi}{b^2}\int_0^\infty\!\rd r\frac{r}{(r^2+i)^3}+O(b^{-5/2})=O\left(\frac{1}{b^2}\right).
\end{equation*}
Thus
\begin{equation*}
\mathcal{J}=\frac{1}{\mathcal{D}^2}\left(\frac{L}{2\pi W}\right)^4\mathcal{S}_2(b)+O\left(\frac{L^2}{W^2}\left(1+\frac{\eta^2}{\omega^3}\right)\right)
\end{equation*}
and
\begin{equation}
\label{trace2}
\mathrm{Tr}\frac{S}{(1-\alpha S)^2}=\frac{1}{\mathcal{D}^2}\left(\frac{L}{2\pi W}\right)^4\mathcal{S}_2(b)+O\left(\frac{L^2}{W^2}\left(\lvert\log{u}\rvert+\frac{\eta^2}{\omega^3}\right)\right),\quad b\gg1.
\end{equation}
\begin{proposition}
\label{ReS2}
The real part of $\mathcal{S}_2(b)$ is given by
\begin{equation*}
\mathrm{Re}\,\mathcal{S}_2(b)=-\frac{1}{b^2}+\frac{2}{b^2}\sum_{n=0}^\infty(-1)^n\left(1-\left(\frac{\pi b/(2n+1)}{\sinh(\pi b/(2n+1))}\right)^{\!2}\right).
\end{equation*}
\end{proposition}
\begin{proof}
Let $\theta$ be the third Jacobi theta function defined by $\theta(t)=\sum_{n\in\mathbb{Z}}e^{-n^2t}$. Then $\mathcal{S}_2(b)$ is the Laplace-Mellin transform of the squared Jacobi theta function
\begin{equation*}
\mathcal{S}_2(b)=\int_0^\infty\!\rd t\sum_{n\in\mathbb{Z}^2}e^{-(n_1^2+n_2^2+ib)t}t=\int_0^\infty\rd t\,\theta(t)^2e^{-ibt}t
\end{equation*}
which can be rewritten in a more convenient way using a Jacobi identity (see e.g. \cite{MR316780})
\begin{equation*}
\theta(t)^2=1+4\sum_{m=1}^\infty\sum_{n=0}^\infty(-1)^ne^{-m(2n+1)t}.
\end{equation*}
Hence
\begin{equation*}
\mathcal{S}_2(b)=-\frac{1}{b^2}+4\sum_{n=0}^\infty\frac{(-1)^n}{(2n+1)^2}\sum_{m=1}^\infty\frac{1}{(m+ib/(2n+1))^2}.
\end{equation*}
To calculate the second sum, we rewrite equation~(\ref{sum1}) as
\begin{equation*}
\frac{1}{z}+i\sum_{m=1}^\infty\left(\frac{1}{m+iz}-\frac{1}{m-iz}\right)=\pi\coth(\pi z)
\end{equation*}
which gives by differentiation with respect to $z$
\begin{equation*}
-\frac{1}{z^2}+\sum_{m=1}^\infty\left(\frac{1}{(m+iz)^2}+\frac{1}{(m-iz)^2}\right)=-\frac{\pi^2}{\sinh^2(\pi z)},
\end{equation*}
or similarly, for real $z$,
\begin{equation}
\label{sum2}
\mathrm{Re}\,\sum_{m=1}^\infty\frac{1}{(m+iz)^2}=\frac{1}{2z^2}-\frac{\pi^2}{2\sinh^2(\pi z)}.
\end{equation}
The result follows by replacing $z$ by $b/(2n+1)$ in (\ref{sum2}).
\end{proof}
\begin{proposition}
\label{propS2}
For $b\gg1$, the real part of $\mathcal{S}_2(b)$ has the following asymptotic behaviour
\begin{equation*}
\mathrm{Re}\,\mathcal{S}_2(b)=-4\pi^2e^{-\pi\sqrt{2b}}b^{-3/4}\sin\left(\pi\sqrt{2b}-\frac{\pi}{8}\right)+O(e^{-\pi\sqrt{2b}}b^{-11/4}).
\end{equation*}
\end{proposition}
\begin{proof}
Given an analytic function $g$ in $\{z\in\mathbb{C}\vert\ \mathrm{Re}(z)\geqslant0\}$ such that
\begin{equation*}
(i)\qquad\lim_{y\rightarrow\infty}\lvert g(x\pm iy)\rvert e^{-2\pi y}=0
\end{equation*}
uniformly in $x$ on every finite inverval in $\left[0,\infty\right.)$ and such that
\begin{equation*}
(ii)\qquad\int_0^\infty\!\rd y\,\lvert g(x+iy)-g(x-iy)\rvert e^{-2\pi y}
\end{equation*}
exists for all $x\geqslant0$ and tends to $0$ as $x\rightarrow\infty$, then Abel-Plana summation formula gives an integral representation of an alternating series through the following relation (see e.g. \cite{MR2793463})
\begin{equation*}
\sum_{n=0}^\infty(-1)^ng(n)=\frac{1}{2}g(0)+i\int_0^\infty\!\rd y\,\frac{g(iy)-g(-iy)}{2\sinh(\pi y)}.
\end{equation*}
Define $g_b(n):=1-\left(\frac{\pi b/(2n+1)}{\sinh(\pi b/(2n+1))}\right)^{\!2}$ which is analytic in $\{z\in\mathbb{C}\vert\ \mathrm{Re}(z)\geqslant0\}$. We have
\begin{align*}
\lvert g_b(x\pm iy)\rvert e^{-2\pi y}&\leqslant\left(1+\frac{(\pi b)^2}{\lvert2x+1\pm2iy\rvert^2}\frac{1}{\lvert\sinh^2(\pi b/(2x+1\pm2iy))\rvert}\right)e^{-2\pi y}\\
&\leqslant\left(1+\frac{(\pi b)^2}{(2x+1)^2+4y^2}\frac{1}{\sinh^2(\pi b(2x+1)/((2x+1)^2+4y^2))}\right)e^{-2\pi y}\\
&\leqslant\left(2+\frac{4y^2}{(2x+1)^2}\right)e^{-2\pi y}\longrightarrow0\quad\textrm{as }y\rightarrow\infty\textrm{ uniformly in }x
\end{align*}
verifying condition $(i)$. Let $\varphi=\pi b(2x+1)/((2x+1)^2+4y^2)$ and $\psi=2\pi by/((2x+1)^2+4y^2)$. Also,
\begin{align*}
&\lvert g_b(x+iy)-g_b(x-iy)\rvert\\
&=(\pi b)^2\Bigg|\frac{1}{(2x+1+2iy)^2}\frac{1}{\sinh^2(\pi b/(2x+1+2iy))}-\frac{1}{(2x+1-2iy)^2}\frac{1}{\sinh^2(\pi b/(2x+1-2iy))}\Bigg|\\
&=\frac{4(\pi b)^2}{((2x+1)^2+4y^2)^2}\frac{1}{(\sinh^2\varphi+\sin^2\psi)^2}\lvert(2x+1)\cosh\varphi\sin\psi-2y\sinh\varphi\cos\psi\rvert\\
&\qquad\qquad\qquad\qquad\qquad\qquad\qquad\qquad\quad\times\lvert(2x+1)\sinh\varphi\cos\psi+2y\cosh\varphi\sin\psi\rvert.
\end{align*}
Using that $\sinh\varphi\geqslant\varphi$, $\sin\psi\leqslant\psi$ and $\cosh\varphi/\sinh^4\varphi\leqslant16\sinh^4(3\varphi/4)$, for all $x,y\geqslant0$, we have
\begin{equation*}
\lvert g_b(x+iy)-g_b(x-iy)\rvert\leqslant\frac{1}{b^2}\left(\gamma_1\frac{y}{2x+1}+\gamma_2\frac{y^3}{(2x+1)^3}\right)\left(\gamma_3+\gamma_4\frac{y^2}{(2x+1)^2}+\gamma_5\frac{y^5}{(2x+1)^5}\right)
\end{equation*}
for some positive constants $\gamma_i$, $i=1,\ldots,5$, which shows that condition $(ii)$ is satisfied. Hence Abel-Plana summation formula implies that
\begin{equation}
\label{S2}
\sum_{n=0}^\infty(-1)^n\left(1-\left(\frac{\pi b/(2n+1)}{\sinh(\pi b/(2n+1))}\right)^{\!2}\right)=\frac{1}{2}-\frac{(\pi b)^2}{2\sinh^2(\pi b)}+i\int_0^\infty\!\rd y\,\frac{g_b(iy)-g_b(-iy)}{2\sinh(\pi y)}.
\end{equation}
We have
\begin{align*}
i\int_0^\infty\!\rd y\,\frac{g_b(iy)-g_b(-iy)}{2\sinh(\pi y)}&=-2\,\mathrm{Im}\int_0^\infty\!\rd y\,\frac{g_b(iy)}{2\sinh(\pi y)}\\
&=8\,\mathrm{Im}\int_0^\infty\!\rd y\,\left(\frac{\pi b}{2iy+1}\right)^{\!2}\frac{e^{-\pi b(2/(2iy+1)+y/b)}}{1+e^{-4\pi b/(2iy+1)}-2e^{-2\pi b/(2iy+1)}}\frac{1}{1-e^{-2\pi y}}\\
&=:8\,\mathrm{Im}\,I(b).
\end{align*}
Let $h(y):=-\pi(2/(2iy+1)+y/b)$ and $$j(y):=\left(\frac{\pi b}{2iy+1}\right)^{\!2}\frac{1}{1+e^{-4\pi b/(2iy+1)}-2e^{-2\pi b/(2iy+1)}}\frac{1}{1-e^{-2\pi y}}.$$ The function $\mathrm{Re}\,h(y)$ attains his maximum at $y_0=\sqrt{b}e^{-i\pi/4}+i/2$ which satisfies $h(y_0)= -\pi\sqrt{2/b}+i(\pi\sqrt{2/b}-\pi/(2b))$ and $h''(y_0)=2\pi b^{-3/2}e^{-3i\pi/4}$. From the saddle point method it follows that
\begin{equation}
\label{S2asympt}
I(b)=\sqrt{\frac{2\pi}{-h''(y_0)}}\frac{1}{\sqrt{b}}e^{bh(y_0)}\left(j(y_0)+O(b^{-1})\right)=\frac{\pi^2}{4}e^{-\pi\sqrt{2b}}e^{i(\pi\sqrt{2b}-\pi/8)}\left(b^{5/4}+O(b^{-3/4})\right).
\end{equation}
Putting Proposition~\ref{ReS2}, (\ref{S2}) and (\ref{S2asympt}) together, it comes
\begin{equation*}
\mathrm{Re}\,\mathcal{S}_2(b)=-4\pi^2e^{-\pi\sqrt{2b}}b^{-3/4}\sin\left(\pi\sqrt{2b}-\frac{\pi}{8}\right)+O(e^{-\pi\sqrt{2b}}b^{-11/4}),\quad b\gg1
\end{equation*}
showing that the function $\mathrm{Re}\,\mathcal{S}_2(b)$ changes sign infinitely many times.
\end{proof}
\begin{proof}[Proof of Theorem~\ref{Th}(ii)]
Putting (\ref{corr}), (\ref{theta}), (\ref{trace2}), Proposition~\ref{propS2} and \cite[Proposition $3.3(ii)$]{erdHos2015altshulerII} together proves Theorem~\ref{Th} $(ii)$.
\end{proof}
Below is a plot of $\mathrm{Re}\,\mathcal{S}_2(b)$.
\begin{figure}[H]
\centering
\includegraphics[width=8cm]{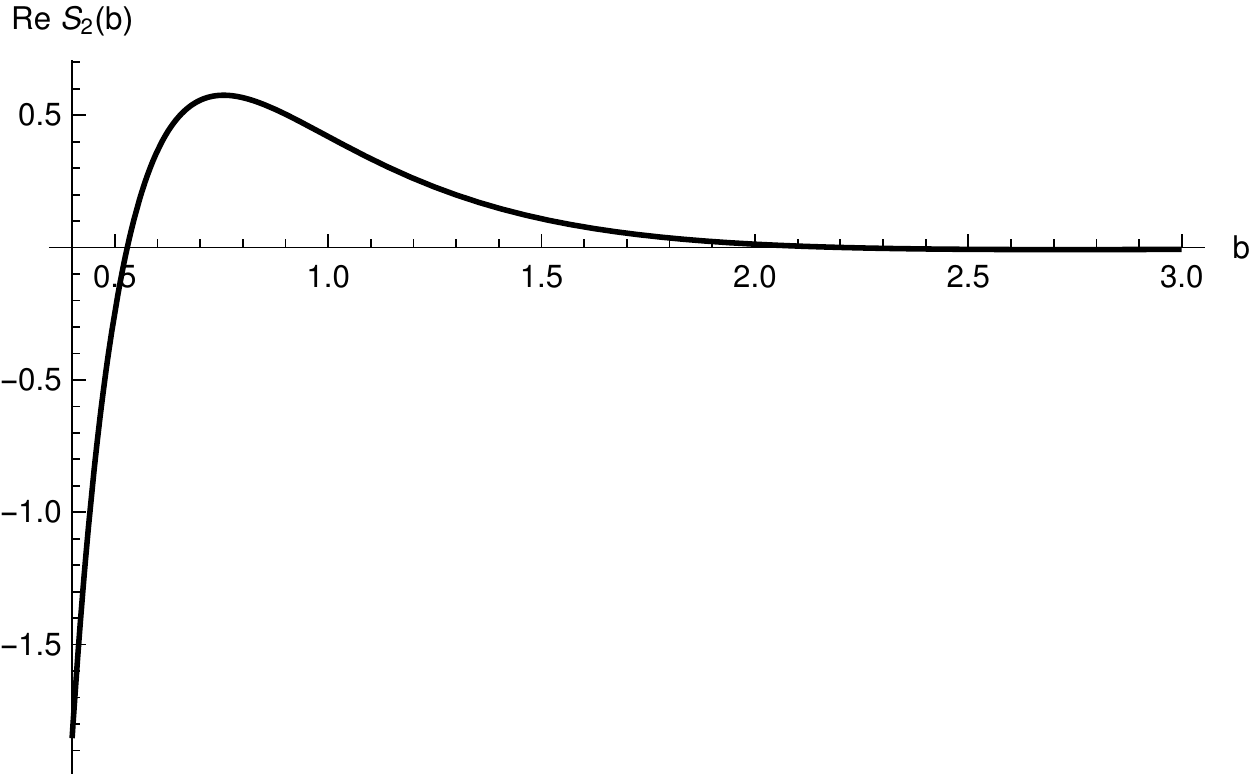}
\end{figure}
And here is a plot of $(4\pi^2)^{-1}b^{3/4}e^{\pi\sqrt{2b}}\,\mathrm{Re}\,\mathcal{S}_2(b)$ compared to its asymptotic expression.
\begin{figure}[H]
\centering
\includegraphics[width=10cm]{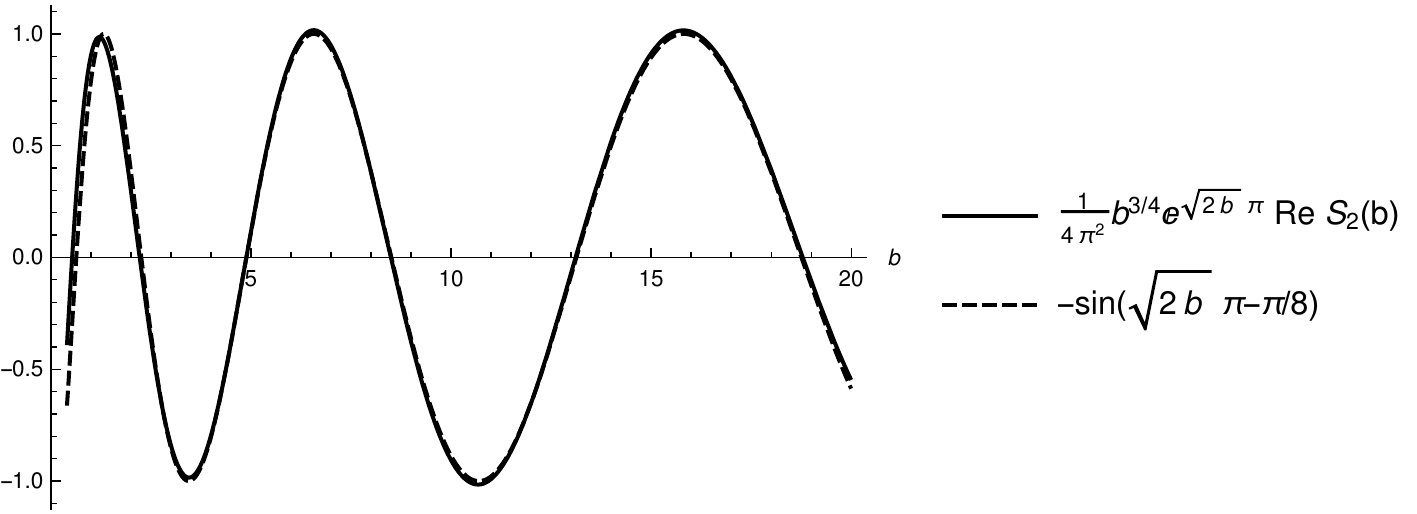}
\end{figure}
\begin{proposition}
\label{sameorder}
The two leading terms of the local density-density correlation function are of the same order for $b\sim(\log(L/W))^2$. More precisely,
\begin{equation*}
\exists\ \gamma\in\left[\frac{1}{2\pi^2},\frac{2}{\pi^2}\right]\textrm{ such that }\frac{L^2}{W^2}b^{-3/2}e^{-\pi\sqrt{2b}}=\lvert\log\omega\rvert\textrm{ where }b=\gamma\left(\log(L/W)\right)^2.
\end{equation*}
\end{proposition}
\begin{proof}
Let $L=W^{1+a}$ and $\omega=W^{-c}$ for some positive constants $a$ and $c$. Using that $b\sim\omega L^2/W^2>1$, it implies that $2a>c$. Write $b=\gamma_0(\log W)^2$ for some $\gamma_0$, then the solution to the equation below
\begin{equation*}
\frac{L^2}{W^2}b^{-3/2}e^{-\pi\sqrt{2b}}=\lvert\log\omega\rvert
\end{equation*}
corresponds to the root of the following function
\begin{equation*}
f(\gamma_0):=W^{2a-\pi\sqrt{2\gamma_0}}-c\gamma_0^{3/2}(\log W)^4
\end{equation*}
which is decreasing in $\gamma_0$ and is such that $f(a^2/(2\pi^2))>0$. Indeed
\begin{equation*}
f\left(\frac{a^2}{2\pi^2}\right)=W^a-\frac{ca^3}{2^{3/2}\pi^3}(\log W)^4>g(a)
\end{equation*}
where $$g(a):=e^{a\beta_1}-\beta_2a^4,\quad\beta_1:=\log W,\quad\beta_2:=\frac{(\log W)^4}{\sqrt{2}\pi^3}.$$Thus $g^\mathrm{\romannumeral 5}(a)=\beta_1^5e^{a\beta_1}>0$ implying that $g^\mathrm{\romannumeral 4}$ is increasing, and $g^\mathrm{\romannumeral 4}(0)=\beta_1^4-24\beta_2=(\log W)^4(1-24/(\sqrt{2}\pi^3))>0$, so that $g^\mathrm{\romannumeral 4}(a)>0$ for all $a>0$. We deduce that $g(a)>0$ for all $a>0$. Also we have
\begin{align*}
f\left(\frac{2a^2}{\pi^2}\right)&=1-\frac{2^{3/2}ca^3}{\pi^3}(\log W)^4\\
&=1-\frac{2^{3/2}}{\pi^3}\log W^c(\log W^a)^3\\
&=1-\frac{2^{3/2}}{\pi^3}\lvert\log\omega\rvert\left(\log\frac{L}{W}\right)^3<0.
\end{align*}
Thus there is a constant $\gamma_0\in\left[\frac{1}{2\pi^2},\frac{2}{\pi^2}\right]$ such that $f(\gamma_0a^2)=0$ showing the proposition.
\end{proof}
As a consequence of the above proposition we deduce the following corollary.
\begin{corollary}
\label{alt}
In dimension $2$, the densities of states are alternately positively and negatively correlated $O(\log(L/W))$ times.
\end{corollary}

\noindent\textbf{Acknowledgements:} The author thanks Antti Knowles for suggesting this problem to her and useful discussions. The author gratefully acknowledges the anonymous referees for a careful reading of the manuscript and valuable comments which greatly improved the readability of the paper.\\

\noindent Conflicts of Interest: The author declares no conflicts of interest.\\

\noindent Data availability: Data sharing is not applicable to this article as no datasets were generated or
analyzed during the current study.

\section*{Appendix}
For the sake of comprehensiveness, we present the main steps of the proof of relation~(\ref{theta}) following \cite{erdHos2015altshulerII}. For more details we refer to Section~$3$ of Erd{\H{o}}s and Knowles' paper. Let introduce the following parameters $\rho\in(0,1/3)$, $\mu$ such that $\rho<\mu<1/3$ and $\delta>0$ satisfying $2\delta<\mu-\rho<3\delta$ and write $\eta:=M^{-\rho}$. We also introduce the following notations taken from \cite{erdHos2015altshulerII}
\begin{equation*}
a_n(t):=\sum_{k\geqslant0}\frac{\alpha_{n+2k}(t)}{(M-1)^k},\quad\alpha_k(t):=2(-i)^k\frac{k+1}{t}J_{k+1}(t),
\end{equation*}
$J_\nu$ denoting the $\nu$-th Bessel function of the first kind. For $n\in\mathbb{N}$, $E\in\mathbb{R}$ and $\phi$ a test function, let $\gamma_n(E)$ and $\widetilde{\gamma}_n(E,\phi)$ be defined by
\begin{equation*}
\gamma_n(E):=\int_0^{\infty}\rd t\,e^{iEt}a_n(t),\quad\widetilde{\gamma}_n(E,\phi):=\int_0^{M^{\rho+\delta}}\rd t\,e^{iEt}\widehat{\phi}(\eta t)a_n(t)
\end{equation*}
$\widehat{\phi}$ being the Fourier transform of $\phi$
\begin{equation*}
\widehat{\phi}(t)=\frac{1}{2\pi}\int_{\mathbb{R}}\rd E\ \!e^{iEt}\phi(E).
\end{equation*}
Then from \cite[Lemma 3.2]{erdHos2015altshulerI}
\begin{equation*}
\gamma_n(E)=\frac{2(-i)^ne^{i(n+1)\arcsin{E}}}{1+(M-1)^{-1}e^{2i\arcsin{E}}}.
\end{equation*}
The leading term $\Theta_{\phi_1,\phi_2}^\eta$ is defined by \cite[(4.60)]{erdHos2015altshulerII}
\begin{equation}
\label{thetadef}
\Theta^\eta_{\phi_1,\phi_2}(E_1,E_2):=\frac{W^d}{L^d}\frac{\mathcal{V}_{\textrm{main}}}{\langle Y_{\phi_1}^\eta(E_1)\rangle\langle Y_{\phi_2}^\eta(E_2)\rangle}
\end{equation}
where the main term $\mathcal{V}_{\textrm{main}}$ is given in \cite[(3.23)]{erdHos2015altshulerII}
\begin{align*}
\mathcal{V}_{\textrm{main}}=\sum_{b_1,b_2\geqslant0}\sum_{(b_3,b_4)\in\mathcal{A}}\mathbf{1}&\left(\sum_{i=1}^4b_i\leqslant M^\mu/2\right)\\
&\times2\mathrm{Re}\,(\widetilde{\gamma}_{2b_1+b_3+b_4}(E_1,\phi_1))\mathrm{Re}\,(\widetilde{\gamma}_{2b_2+b_3+b_4}(E_2,\phi_2))\mathcal{I}^{b_1+b_2}\mathrm{Tr}\,S^{b_3+b_4}
\end{align*}
where $\mathcal{A}$ is the set $\mathcal{A}:=(\{1,2,\ldots\}\times\{0,1,\ldots\})\backslash\{(2,0),(1,1)\}$. Let $\psi_i(E):=\phi_i(-E)$, $i=1,2$, where $\phi_i$, $i=1,2$, are test functions satisfying conditions \textbf{(C)}. It is shown that the above expands as \cite[(3.64)]{erdHos2015altshulerII}
\begin{align}
\label{vmain}
\mathcal{V}_{\textrm{main}}=\sum_{b_1,b_2=0}^{\lfloor M^\mu\rfloor-1}\sum_{(b_3,b_4)\in\mathcal{A}_\mu}2\,\textrm{Re}\,(\gamma_{2b_1+b_3+b_4}\ast\psi_1^{\leqslant,\eta})(E_1)\,2\,\textrm{Re}\,(\gamma_{2b_2+b_3+b_4}&\ast\psi_2^{\leqslant,\eta})(E_2)\,\mathcal{I}^{b_1+b_2}\textrm{Tr}\,S^{b_3+b_4}\nonumber\\
&+O_q(L^dM^{-q})
\end{align}
where $\psi_i^{\leqslant,\eta}(E):=\eta^{-1}\psi_i(\eta^{-1}E)\chi(M^{-\tau/2}E)$, $i=1,2$, $\chi$ being a smooth non-negative symmetric function bounded by $1$ satisfying $\chi(E)=1$ for $\lvert E\rvert\leqslant1$ and $\chi(E)=0$ for $\lvert E\rvert\geqslant2$ and $\tau$ is a positive constant such that $\eta\leqslant M^{-\tau}\omega$. The set $\mathcal{A}_\mu$ is the subset $\mathcal{A}_\mu:=(\{1,2,\ldots,\lfloor M^\mu\rfloor\}\times\{0,1,\ldots,\lfloor M^\mu\rfloor-1\})\backslash\{(2,0),(1,1)\}$ and the convolution of two functions $\phi$ and $\psi$ is defined by
\begin{equation*}
(\phi\ast\psi)(E):=\frac{1}{2\pi}\int\rd E'\phi(E-E')\psi(E').
\end{equation*}
Using relation $(2\,\textrm{Re}\,x_1)(2\,\textrm{Re}\,x_2)=2\,\textrm{Re}\,(x_1\overline{x}_2+x_1x_2)$ in (\ref{vmain}), the authors in \cite{erdHos2015altshulerII} then split $\mathcal{V}_{\textrm{main}}$ into two parts as $\mathcal{V}_{\textrm{main}}=2\,\textrm{Re}\,(\mathcal{V}'_{\textrm{main}}+\mathcal{V}''_{\textrm{main}})+O_q(L^dM^{-q})$ where $\mathcal{V}'_{\textrm{main}}$ identifies with the $x_1\overline{x}_2$ part and $\mathcal{V}''_{\textrm{main}}$ with the $x_1x_2$ part, with $x_1=\gamma_{2b_1+b_3+b_4}\ast\psi_1^{\leqslant,\eta}$ and $x_2=\gamma_{2b_2+b_3+b_4}\ast\psi_2^{\leqslant,\eta}$. On one hand they show in \cite[(3.76)]{erdHos2015altshulerII} that
\begin{equation*}
\lvert\mathcal{V}''_{\textrm{main}}\rvert\leqslant \frac{CL^d}{M}
\end{equation*}
and on the other hand they split again $\mathcal{V}'_{\textrm{main}}$ into two terms $\mathcal{V}'_{\textrm{main}}=\mathcal{V}'_{\textrm{main},0}-\mathcal{V}'_{\textrm{main},1}$ where \cite[(3.69)]{erdHos2015altshulerII}
\begin{equation*}
\mathcal{V}'_{\textrm{main},1}=O\left(\frac{L^d}{M}\right).
\end{equation*}
Then $\mathcal{V}'_{\textrm{main},0}$ is shown to be equal to \cite[(3.73)]{erdHos2015altshulerII}
\begin{align*}
\mathcal{V}'_{\textrm{main},0}=\left[T(E_1)\overline{T(E_2)}\frac{e^{iA_1}}{1+e^{2iA_1}\mathcal{I}}\frac{e^{-iA_2}}{1+e^{-2iA_2}\mathcal{I}}\textrm{Tr}\frac{e^{i(A_1-A_2)}S}{(1-e^{i(A_1-A_2)}S)^2}\right]\ast\psi_1^{\leqslant,\eta}(E_1)\ast\psi_2^{\leqslant,\eta}(E_2)+O\left(\frac{L^d}{M}\right)
\end{align*}
where $A_i:=\arcsin{E_i}$ and
\begin{equation*}
T(z):=\frac{2}{1+(M-1)^{-1}e^{2i\arcsin{z}}}.
\end{equation*}
Using the estimate
\begin{equation*}
\frac{e^{\pm iA_i}}{1+e^{\pm2iA_1}\mathcal{I}}=\frac{1}{\pi\nu}\left(1+O(M^{-1}+\omega)\right),\ i=1,2,\textrm{ where }\nu=\frac{2}{\pi}\sqrt{1-E^2}
\end{equation*}
and by definition of the convolution product, we have
\begin{equation*}
\mathcal{V}'_{\textrm{main},0}=\frac{1}{\pi^4\nu^2}\int_{\mathbb{R}^2}\rd v_1\rd v_2\psi_1^{\leqslant,\eta}(v_1)\psi_2^{\leqslant,\eta}(v_2)\textrm{Tr}\frac{e^{i(\arcsin(E_1-v_1)-\arcsin(E_2-v_2))}S}{(1-e^{i(\arcsin(E_1-v_1)-\arcsin(E_2-v_2))}S)^2}\left(1+O(M^{-1}+\omega)\right)+O\left(\frac{L^d}{M}\right).
\end{equation*}
On the domain of integration we have
\begin{align*}
&\arcsin(E_1-v_1)=\arcsin{E}-\left(\frac{\omega}{2}+v_1\right)\frac{1}{\sqrt{1-E^2}}+O(\omega(\omega+M^{-\tau/2}))\\
&\arcsin(E_2-v_2)=\arcsin{E}+\left(\frac{\omega}{2}-v_2\right)\frac{1}{\sqrt{1-E^2}}+O(\omega(\omega+M^{-\tau/2}))
\end{align*}
which implies
\begin{equation*}
e^{i(\arcsin(E_1-v_1)-\arcsin(E_2-v_2))}=1-\frac{2i}{\pi\nu}(\omega+v_1-v_2)+O(\omega(\omega+M^{-\tau/2})).
\end{equation*}
Moreover by definition of $\psi_i^{\leqslant,\eta}$ and the test functions, we have for $i=1,2$
\begin{equation*}
M^{\tau/2}\int_{-1}^1\rd v\,\psi_i^\eta(M^{\tau/2}v)=\int_\mathbb{R}\rd v\,\psi(v)+O_q\big(M^{-(\tau/2+q)}\big)=2\pi+O_q\big(M^{-(\tau/2+q)}\big),\ \forall q>0
\end{equation*}
and
\begin{equation*}
M^{\tau/2}\int_1^2\rd v\,\psi_i^\eta(M^{\tau/2}v)\chi(v)\leqslant\int_{M^{\tau/2+\rho}}^{2M^{\tau/2+\rho}}\rd v\,\psi(v)=O_q\big(M^{-(\tau/2+q)}\big),\ \forall q>0
\end{equation*}
implying that
\begin{equation*}
\int_\mathbb{R}\rd v\,\psi_i^{\leqslant,\eta}(v)=M^{\tau/2}\int_{-1}^1\rd v\,\psi_i^\eta(M^{\tau/2}v)+2M^{\tau/2}\int_1^2\rd v\,\psi_i^\eta(M^{\tau/2}v)\chi(v)=2\pi+O_q\big(M^{-(\tau/2+q)}\big),\ \forall q>0.
\end{equation*}
Also,
\begin{equation*}
\int_{\mathbb{R}^2}\rd v_1\rd v_2\,\psi_1^{\leqslant,\eta}(v_1)\psi_2^{\leqslant,\eta}(v_2)(v_1-v_2)=0.
\end{equation*}
Thus $\mathcal{V}'_{\textrm{main},0}$ becomes
\begin{equation*}
\mathcal{V}'_{\textrm{main},0}=\frac{4}{\pi^2\nu^2}\alpha\textrm{Tr}\frac{S}{(1-\alpha S)^2}+O\left(\frac{L^d}{W^d}\right)
\end{equation*}
and
\begin{equation}
\label{vmainfinal}
\mathcal{V}_{\textrm{main}}=\frac{8}{\pi^2\nu^2}\textrm{Re}\,\textrm{Tr}\frac{S}{(1-\alpha S)^2}(1+O(\omega))+O\left(\frac{L^d}{W^d}\right)
\end{equation}
since $\alpha=1+O(\omega)$. Recall from \cite[Lemma 4.17]{erdHos2015altshulerII} that the expectation value of the local density of states around energy $E\in\left[-1+\kappa,1-\kappa\right]$ is given by
\begin{equation}
\label{expY}
\langle Y_{\phi_i}^\eta(E)\rangle=2\pi\nu+O(\eta).
\end{equation}
Putting equations (\ref{thetadef}), (\ref{vmainfinal}) and (\ref{expY}) together ends the proof of (\ref{theta}).

We emphasise that in the above proof, which is derived in \cite{erdHos2015altshulerII}, is valid in all regimes, i.e. while the authors are considering the diffusive regime in their paper, the hypothesis $\eta\gg\eta_c$ is not used at any stage of it and is thus valid in our case.


\bibliographystyle{plain}
\bibliography{bibliography}

\end{document}